%% file: TSFtypeBF_axriv.tex
\documentclass[english]{article}
\usepackage[T1]{fontenc}
\usepackage[latin9]{inputenc}
\usepackage{babel}
\usepackage{array}
\usepackage{url}
\usepackage{multirow}
\usepackage{amsmath}
\usepackage{amsthm}
\usepackage{graphicx}
\usepackage{epstopdf}
\usepackage{epsfig}
\usepackage[authoryear]{natbib}
\usepackage[unicode=true,pdfusetitle,
bookmarks=true,bookmarksnumbered=false,bookmarksopen=false,
breaklinks=false,pdfborder={0 0 1},backref=false,colorlinks=false]
{hyperref}
\hypersetup{
	colorlinks,urlcolor=blue,linkcolor=blue,anchorcolor=blue,citecolor=blue}
\usepackage{breakurl}

\makeatletter

\theoremstyle{plain}
\newtheorem{theorem}{\protect\theoremname}
\theoremstyle{remark}
\newtheorem{remark}{\protect\remarkname}

\newcommand{\be}{\begin{equation}}
	\newcommand{\ee}{\end{equation}}
\newcommand{\bey}{\begin{eqnarray}}
	\newcommand{\eey}{\end{eqnarray}}
\newcommand{\beyn}{\begin{eqnarray*}}
	\newcommand{\eeyn}{\end{eqnarray*}}
\input{paperconfig.tex}
\usepackage[noblocks]{authblk}
\def\singlespace{\def\baselinestretch{1}\@normalsize}

\makeatother

\providecommand{\remarkname}{Remark}
\providecommand{\theoremname}{Theorem}

\usepackage{xr}

\usepackage{color}%\textcolor{red}{}
\usepackage{soul,xcolor}
\setstcolor{red}%\st{Some overstruck text}
\newif\ifApproveEdit
\ApproveEditfalse
\ifApproveEdit

\newcommand{\del}[1]{\iffalse{#1}\fi}
\else

\newcommand{\del}[1]{\st{#1}}
\fi

\global\long\def\b#1{{\bf \bm{\mathit{#1}}}}
\global\long\def\0{\mbox{\b 0}}
\global\long\def\bn{\b n}
\global\long\def\bh{\b h}
\global\long\def\bu{\b u}

\global\long\def\bx{\b x}
\global\long\def\by{\b y}

\global\long\def\bz{\b z}

\global\long\def\bI{\b I}
\global\long\def\bJ{\b J}
\global\long\def\bD{\b D}
\global\long\def\bR{\b R}
\global\long\def\bmu{\b {\mu}}
\global\long\def\bSigma{\b{\Sigma}}
\global\long\def\bGamma{\b {\Gamma}}
\global\long\def\bOmega{\b {\Omega}}

\global\long\def\barbx{\bar{\bx}}
\global\long\def\barby{\bar{\by}}
\global\long\def\h#1{\hat{#1}}
\global\long\def\hbSigma{\h{\bSigma}}
\global\long\def\hbOmega{\h{\bOmega}}
\global\long\def\iidsim{\stackrel{\text{i.i.d.}}{\sim}}
\global\long\def\tr{\operatorname{tr}}
\global\long\def\E{\operatorname{E}}
\global\long\def\Cov{\operatorname{Cov}}
\global\long\def\Var{\operatorname{Var}}
\global\long\def\ARE{\operatorname{ARE}}
\global\long\def\diag{\operatorname{diag}}
\renewcommand{\H}{\mbox{H}}
\newcommand{\convP}{\stackrel{P}{\longrightarrow}}
\newcommand{\convL}{\stackrel{L}{\longrightarrow}}
\global\long\def\dequ{\stackrel{d}{=}}

\newcommand{\calN}{\mathcal{N}}
\newcommand{\calK}{\mathcal{K}}

\newcommand{\tF}{\tilde{F}}

\newcommand{\hd}{\hat{d}}
\newcommand{\halpha}{\hat{\alpha}}

\begin{document}
	%\begin{frontmatter}
	
	\author[a]{\small Tianming Zhu}
	\author[a]{\small Pengfei Wang}
	\author[b]{\small Jin-Ting Zhang}
	\affil[a]{\footnotesize  National Institute of Education, Nanyang Technological University, %1 Nanyang Walk,
		Singapore}
	\affil[b]{\footnotesize Department of Statistics and Data Science, National University of Singapore, %\protect\\3 Science Drive 2, Singapore 117546,
		Singapore}
	
	\title{\Large  Two-sample Behrens--Fisher problems for high-dimensional data: a normal reference $F$-type test} %\tnoteref{t1}}
%\tnotetext[t1]{The work was financially supported by the National University of Singapore Academic Research grant R-155-000-212-114.}

%		\author{Tianming Zhu} %\thanks{stazt@nus.edu.sg}}
%		\author{Jin-Ting Zhang\corref{cor1}}
%		\ead{stazjt@nus.edu.sg}
%		\cortext[cor1]{Department of Statistics and Applied Probability, National University of Singapore, Singapore 117546,
%			Singapore}
%		\address{Department of Statistics and Applied Probability, National University of Singapore, Singapore 117546,
%			Singapore}
\maketitle

\vspace{-1cm}

\abstract{	The problem of testing the equality of mean vectors for high-dimensional data  has been intensively  investigated in the literature.  However, most of the existing  tests impose strong assumptions on the underlying group covariance matrices which  may not be satisfied or hardly be checked in practice. In this article, an $F$-type test for two-sample Behrens--Fisher problems for high-dimensional data is  proposed and studied. When the two samples are normally distributed and when the null hypothesis is valid,  the proposed  $F$-type test statistic is shown to be  an $F$-type mixture,  a ratio of two independent $\chi^2$-type mixtures. Under some regularity conditions and the null hypothesis, it is shown that the proposed $F$-type test statistic and the above $F$-type mixture have the same normal and  non-normal limits. It is then  justified  to approximate the null distribution of the proposed $F$-type  test statistic by that of the $F$-type mixture, resulting in the so-called normal reference $F$-type test. Since the $F$-type mixture is a ratio of two independent $\chi^2$-type mixtures, we employ the Welch--Satterthwaite $\chi^2$-approximation to the distributions of the numerator and the denominator of the $F$-type mixture respectively, resulting in an approximation $F$-distribution whose degrees of freedom can be  consistently estimated from the data. The asymptotic power of the proposed $F$-type  test is established.  Two simulation studies are conducted and they show that in terms of size control,  the proposed   $F$-type test outperforms two existing competitors. The proposed $F$-type test is also illustrated by a real data example.}

	\noindent{\bf KEY WORDS}:
High-dimensional Behrens--Fisher problem; $F$-type test; $\chi^2$-type mixtures;  $F$-type mixture; Welch--Satterthwaite $\chi^2$-approximation.
\maketitle


\section{Introduction}\label{Intro.sec}
Nowadays, large amount of information is frequently  recorded and stored for analytics purpose. Applications of high-dimensional data have been found in various domains including  economics, genetics, pharmacy, medicine and so on. This paper is motivated by a  coronavirus disease 2019 (COVID-19, also known as SARS-COV-2) data set described and studied in \cite{thair2021transcriptomic}.   As an novel coronavirus, COVID-19 has developed into a global pandemic and affected  millions of people during the last three years. However, the actual knowledge about COVID-19 is still limited. \cite{thair2021transcriptomic} profiled peripheral blood from 24 healthy controls and 62 prospectively enrolled patients with community-acquired lower respiratory tract infection by SARS-COV-2 within the first 24 hours of hospital admission using RNA sequencing. It is of interest and worthwhile to check whether those prospectively enrolled patients with COVID-19 and healthy controls have the same mean RNA sequencing transcriptome profiles. Each RNA sequencing transcriptome profile has $20,460$ measurements. This means that the dimension of a datum point  is  much larger than the total sample size $86$. Thus, the above problem is a two-sample  problem for high-dimensional data. For high-dimensional data,    it is generally very  difficult to check whether the  underlying covariance matrices of the two samples  are equal. Thus, the above problem is also called  a two-sample Behrens--Fisher (BF) problem for high-dimensional data.

Mathematically, a two-sample BF problem for high-dimensional data can be described as follows. Suppose we have two independent high-dimensional samples:
\begin{equation}\label{2samp.equ}
	\by_{i1},\ldots,\by_{in_i} \mbox{ are i.i.d. with }\E(\by_{i1})=\bmu_i, \Cov(\by_{i1})=\bSigma_i, i=1,2,
\end{equation}
where the dimension of the observations $\by_{ij}$ is $p$,  which is close to or even much larger than the total sample size $n=n_1+n_2$. Of interest is to  test whether the two mean vectors $\bmu_1$ and $\bmu_2$ are equal:
\begin{equation}\label{H0.equ}
	\H_0: \bmu_1=\bmu_2\; \mbox{ vs }\; \H_1: \bmu_1\neq\bmu_2,
\end{equation}
without  assuming that the two covariance matrices are equal, i.e.,  $\bSigma_1=\bSigma_2$.

When $p=1$, the above problem reduces to the well-known two-sample BF problem for univariate data, which has been studied by a number of authors in the past few decades, including \cite{fisher1935fiducial, fisher1939comparison,scheffe1970practical,tang2007distributional,zhang2013IJASP,liuguozhouzhang2016SRL} among others. When $p$ is small and fixed, the above problem reduces to the so-called two-sample BF problem for multivariate data, which has also  been investigated by a number of authors in the past few decades,  including
\cite{james1954tests,yao1965approximate, johansen1980welch,zhang2011tech, zhang2012open, zhangzhouguoliu2016JAS} among others. Many of these works  are based on the following classical Wald-type test statistic:
\be\label{Tn.equ}
T^2_W=\frac{n_1n_2}{n}(\barby_1-\barby_2)^\top\hbOmega_n^{-1}(\barby_1-\barby_2),
\ee
where $\barby_1$  and $\barby_2$ are the sample mean vectors $\barby_i=n_i^{-1}\sum_{j=1}^{n_i}\by_{ij},i=1,2,$
and with
\be\label{hSigma.sec1}
\hbSigma_i=(n_i-1)^{-1}\sum_{j=1}^{n_i}(\by_{ij}-\barby_i)(\by_{ij}-\barby_i)^\top,\; i=1,2,
\ee
being  the usual sample covariance matrices for the two samples,
\be\label{hbOmega.equ}
\hbOmega_n=\frac{n_2}{n}\hbSigma_1+\frac{n_1}{n}\hbSigma_2,
\ee
is the usual unbiased estimator of the covariance matrix $\bOmega_n$ of $(n_1n_2/n)^{1/2}(\barby_1-\barby_2)$, i.e.,
\be\label{bOmega.equ}
\bOmega_n=\Cov\Big[(n_1n_2/n)^{1/2}(\barby_1-\barby_2)\Big]=\frac{n_2}{n}\bSigma_1+\frac{n_1}{n}\bSigma_2.
\ee
In the high-dimensional scenario, we often have $p>n$ which  means that the covariance matrices  $\hbSigma_1$ and $\hbSigma_2$ are  singular so that  the above mentioned  $T^2_W$-based   tests  are no longer applicable. In fact,  even when $p$ is smaller than but close to  the total sample size $n$, as seen from the  simulation results presented  in \cite{ZhangTSBF2020}, these $T_W^2$-based BF tests  may not have a good size control  and are less  powerful.  This shows that these traditional  BF tests for multivariate data  are less useful in high-dimensional settings.

To overcome the above difficulties, much work has been done for the two-sample  BF problems for high-dimensional data  in the past decade.  For example,  \cite{chen2010two} proposed  an $L^2$-norm  based approach using   the following U-statistics-based  test statistic
\be\label{CQstat}
T_{CQ}=\frac{\sum_{i\neq j}^{n_1}\by_{1i}^\top \by_{1j}}{n_1(n_1-1)}+\frac{\sum_{i\neq j}^{n_2}\by_{2i}^\top \by_{2j}}{n_2(n_2-1)}-2\frac{\sum_{i=1}^{n_1}\sum_{j=1}^{n_2}\by_{1i}^\top\by_{2j}}{n_1n_2}.
\ee
When both $n$ and $p$ are large, however, it is very time-consuming to compute  $T_{CQ}$ as demonstrated in \cite{zhangguozhoucheng2020JASA}. By some simple algebra, fortunately, we can
rewrite $T_{CQ}$ as $T_{CQ}=\|\barby_1-\barby_2\|^2-n_1n_2n^{-1}\tr(\hbOmega_n)$ so that $T_{CQ}$ can be computed much more quickly where $\hbOmega_n$ is defined in (\ref{hbOmega.equ}) and $\tr(\hbOmega_n)$ denotes the trace of $\hbOmega_n$.  Notice that \cite{chen2010two} constructed  $T_{CQ}$  to estimate $\|\bmu_1-\bmu_2\|^2$ unbiasedly and they imposed strong assumptions  so that $T_{CQ}$ is asymptotically normally distributed. In real data analysis,  $T_{CQ}$ is often   conducted by  normal approximation to its null distribution without checking if the required assumptions are satisfied. However, when the required  assumptions
as those imposed in  \cite{chen2010two} are not satisfied, \cite{ZhangTSBF2020} showed that $T_{CQ}$  may not be asymptotically  normally distributed and  the test results obtained by $T_{CQ}$  may then be less reliable; see Section~\ref{appl.sec} for some details.

To overcome this problem,  \cite{ZhangTSBF2020}  proposed a normal reference test using the following   $L^2$-norm based test statistic:
\begin{eqnarray}
	T_{n,p}=\frac{n_1n_2}{n}\|\barby_1-\barby_2\|^2, \label{Tstat.sec1}
\end{eqnarray}
which is proportional to the squared $L^2$-norm of the sample mean difference vector $\barby_1-\barby_2$. \cite{ZhangTSBF2020} showed that under some mild conditions, $T_{n,p}$ and a chi-square-type mixture have the same asymptotic normal or non-normal  distributions. It is then justified to approximate  the null distribution of $T_{n,p}$  using that of the  $\chi^2$-type mixture whose distribution can be well approximated  using the well-known Welch--Satterthwaite $\chi^2$-approximation with approximation parameters consistently estimated from the data \cite{satterthwaite1946approximate,welch1947generalization,zhangguozhoucheng2020JASA}.

Notice that in \cite{ZhangTSBF2020},  to conduct the normal reference  test $T_{n,p}$, the variation of $T_{n,p}$ has  not been  taken into account. This means that when the total  sample size $n$ is small, the size control of $T_{n,p}$ will be  less accurate. This is actually verified by  some simulation results presented  in \cite{ZhangTSBF2020} where   for some simulation  cases with small total sample sizes, in terms of size control,  $T_{n,p}$  is still somewhat liberal. This gives a good motivation for us to take the variation of $T_{n,p}$ into account to construct a test statistic with  a better size control than $T_{n,p}$. To this end, we mimic the construction of the classical $F$-test statistic in  univariate and multivariate data analysis, see for example \cite{anderson2003introduction},  and  propose the following $F$-type test statistic
\begin{equation}\label{Fnp.equ}
	F_{n,p}=\frac{T_{n,p}}{S_{n,p}}=\frac{n_1n_2n^{-1}\|\barby_1-\barby_2\|^2}{\tr(\hbOmega_n)},
\end{equation}
where $T_{n,p}$ is defined in (\ref{Tstat.sec1}) and with $\hbOmega_n$ defined in (\ref{hbOmega.equ}), 
\be\label{Snp.sec1}
S_{n,p}=\tr(\hbOmega_n)=\frac{n_2}{n}\tr(\hbSigma_1)+\frac{n_1}{n}\tr(\hbSigma_2),
\ee
is an  unbiased estimator of
\be\label{trOmega.sec1}
\E(T_{n,p})=\tr(\bOmega_n)=\frac{n_2}{n}\tr(\bSigma_1)+\frac{n_1}{n}\tr(\bSigma_2).
\ee
The construction of the above $F$-type test statistic $F_{n,p}$  guarantees that under the null hypothesis,  the numerator $T_{n,p}$ and the denominator $S_{n,p}$ of $F_{n,p}$  have the same expectation. Notice that $F_{n,p}$ can also be obtained from the Wald-type test statistic (\ref{Tn.equ})  via  replacing the covariance matrix  $\hbOmega_n$  with the  matrix  $\tr(\hbOmega_n)\bI_p$, partially taking the variation of $\sqrt{n_1n_2/n}(\by_1-\by_2)$ into account.

The main idea for constructing the $F$-type test statistic for high-dimensional hypothesis testing may be dated back to \cite{dempster1958high,dempster1960signi} although his non-exact test statistic is constructed in a complicated way and under the  Gaussian and equal-covariance matrix  assumptions. After some complicated derivation, \cite{dempster1958high,dempster1960signi}  showed that his non-exact test statistic can be well approximated by an $F$-distribution with $r$ and $(n-2)r$ degrees of freedom where $r$ is some approximation parameter which can be consistently estimated from the data. \cite{bai1996effect} pioneerly  investigated Dempster's non-exact test and showed that under some regularity conditions, Dempster's non-exact test statistic is asymptotically normally distributed without imposing  the restricted  Gaussian assumption. \cite{Srivastava2006JMV} extended Dempster's non-exact test for high-dimensional MANOVA and showed that under some regularity conditions,  the resulting non-exact test statistic is also asymptotically normally distributed without imposing the restricted Gaussian assumption. Nevertheless, as seen from \cite{bai1996effect} and \cite{Srivastava2006JMV},    the asymptotic normality of these non-exact tests is guaranteed only when some strong assumptions are imposed on the underlying common covariance matrix of the  high-dimension samples. This means that when these strong  assumptions are not satisfied, these non-exact tests will be less accurate in terms of size control.

To overcome the  difficulties mentioned above, in this paper, we  propose a normal reference approach which may be briefly described as follows. For convenience,   let $F_{n,p,0}$ denote the test statistic $F_{n,p}$ (\ref{Fnp.equ}) under the null hypothesis  and further let $F_{n,p,0}^*$ denote $F_{n,p,0}$ when the data are Gaussian. We call the distribution of $T_{n,p,0}^*$ as the normal reference distribution of $F_{n,p,0}$ and show that the distribution of $F_{n,p,0}^*$ is an $F$-type mixture (\ref{Fnp0star.equ})  as described in \cite[Sec. 4.4]{zhang2013analysis}.  The core idea of the normal reference approach for $F_{n,p}$ is to  approximate  the distribution of $F_{n,p,0}$ using that of $F_{n,p,0}^*$.  We justify this normal reference approach via showing that under some regularity conditions, $F_{n,p,0}$ and $F_{n,p,0}^*$ have the same normal or non-normal asymptotic distributions (see Theorem~\ref{limit.thm} of Section~\ref{main.sec}  for details). Following \cite{dempster1958high,dempster1960signi}, we can further approximate the numerator and denominator of the $F$-type mixture $F_{n,p,0}^*$ using  the Welch--Satterthwaite $\chi^2$-approximation respectively so that the $F$-type mixture $F_{n,p,0}^*$  can be well approximated with an approximation $F$-distribution with its degrees of freedom estimated consistently from the data; see Section~\ref{main.sec} for more details. Compared with \cite{dempster1958high,dempster1960signi, bai1996effect, Srivastava2006JMV}'s non-exact tests, our $F$-type  test works without the Gaussian and equal-covariance matrix assumptions and we approximate the null distribution of $F_{n,p}$ using an $F_{d_1,d_2}$-distribution instead of a normal distribution for normal and non-normal high-dimensional data. The degrees of freedom of the $F_{d_1,d_2}$-distribution can be consistently estimated from the data.  Simulation studies conducted in Section~\ref{simu.sec} indicate that in terms of size control, our $F$-type  test outperforms the tests by \cite{ZhangTSBF2020} and \cite{chen2010two} generally.

The rest of this paper is organized as follows. The main results are presented  in Section~\ref{main.sec} where the asymptotic distributions of $F_{n,p,0}$ and $F_{n,p,0}^*$ are derived,
the methods for approximating the null distribution of $F_{n,p}$  are described, and  the asymptotic power of $F_{n,p}$  under a local alternative is established. Two simulation studies and applications to the COVID-19 data set are presented
in Sections~\ref{simu.sec} and~\ref{appl.sec}, respectively. We give some concluding remarks in Section~\ref{remark.sec} and leave the technical proofs of the main results   in the Appendix.
	\section{Main Results}\label{main.sec}
\subsection{Asymptotic null distribution}

We first  derive the null distribution of $F_{n,p}$ (\ref{Fnp.equ}). For this purpose, we set
\be\label{csamp.sec2}
\bx_{ij}=\by_{ij}-\bmu_i,j=1,\ldots,n_i;i=1,2,
\ee
as the centering two samples obtained from the two original samples (\ref{2samp.equ}) after subtracting the associated group mean vectors.  It follows that the sample mean vectors and sample covariance matrices of the centering two samples (\ref{csamp.sec2}) are given by
$\barbx_i=\barby_i-\bmu_i,i=1,2$ and $\hbSigma_i=(n_i-1)^{-1}\sum_{j=1}^{n_i}(\bx_{ij}-\barbx_i)(\bx_{ij}-\barbx_i)^\top,i=1,2$ respectively.
Set
\begin{equation*}\label{Fnp0.equ}
	F_{n,p,0}=\frac{T_{n,p,0}}{S_{n,p,0}}=\frac{n_1n_2n^{-1}\|\barbx_1-\barbx_2\|^2}{\tr(\hbOmega_n)},
\end{equation*}
where $S_{n,p,0}=\tr(\hbOmega_n)$ is based on the centering two samples (\ref{csamp.sec2}) but it is the same as $S_{n,p}=\tr(\hbOmega_n)$  (\ref{Snp.sec1}) based on the original two samples (\ref{2samp.equ}). Notice also that $\E(\hbOmega_n)=\bOmega_n$ which is  defined in (\ref{bOmega.equ}).
It is easy to see that the distribution of $F_{n,p,0}$ is the same as the null distribution of $F_{n,p}$. Therefore, studying the null distribution of $F_{n,p}$ is equivalent to studying the distribution of $F_{n,p,0}$.

Throughout this paper, let $\chi_v^2$ denote a central chi-square distribution with $v$ degrees of freedom.  When the two samples (\ref{2samp.equ}) are normally distributed, it is easy to see that
for any given $n$ and $p$, the distribution of $T_{n,p,0}$ has the same distribution as that of the following  $\chi^2$-type mixture:
\be\label{Tnp0star.equ}
T_{n,p,0}^*\dequ \sum_{r=1}^p\lambda_{n,p,r}A_r, \;\; A_r\iidsim \chi_1^2,
\ee
where $\dequ$ denotes equality in distribution,  and  the distribution of $S_{n,p,0}=\tr(\hbOmega_n)$ has the same distribution as that of the following chi-square-type mixture:
\be\label{Snp0star.equ}
\begin{split}
&S^*_{n,p,0}\dequ \frac{n_2}{n(n_1-1)}\sum_{r=1}^p\lambda_{1r}B_{1r}+\frac{n_1}{n(n_2-1)}\sum_{r=1}^p\lambda_{2r}B_{2r},\\
& B_{1r}\iidsim \chi_{n_1-1}^2,\;B_{2r}\iidsim \chi_{n_2-1}^2,
\end{split}
\ee
where $\lambda_{n,p,1},\ldots,\lambda_{n,p,p}$ are the eigenvalues of the covariance matrix $\bOmega_n$ as defined in (\ref{bOmega.equ}) and $\lambda_{1r},r=1,\ldots,p$ and $\lambda_{2r},r =1,\ldots,p$ are the eigenvalues of $\bSigma_1$ and $\bSigma_2$, respectively. Note that under the Gaussian assumption, $T_{n,p,0}$ and $S_{n,p,0}$ are independent. Therefore, when the two samples (\ref{2samp.equ}) are normally distributed, for any given $n$ and $p$, the distribution of $F_{n,p,0}$ has the same distribution as that of $F_{n,p,0}^*$:
\begin{equation}\label{Fnp0star.equ}
	F_{n,p,0}^*=\frac{T_{n,p,0}^*}{S^*_{n,p,0}}\dequ\frac{\sum_{r=1}^p\lambda_{n,p,r}A_r}{\left[(n_1 -1)^{-1}n_2\sum_{r=1}^p\lambda_{1r}B_{1r}+(n_2-1)^{-1}n_1\sum_{r=1}^p\lambda_{2r}B_{2r}\right]/n},
\end{equation}
where $A_r,r=1,\ldots,p$, $B_{1r},r=1,\ldots,p$ and $B_{2r},r=1,\ldots,p$ are mutually independent. Notice that $F_{n,p,0}^*$  is an $F$-type mixture, i.e.,  a ratio of two independent $\chi^2$-type mixtures as defined in \cite[Sec. 4.4]{zhang2013analysis}.

As mentioned in the introduction section,  $F_{n,p,0}^*$ denotes  $F_{n,p,0}$ when the two samples (\ref{2samp.equ}) are normally distributed, and  we call the distribution of $F_{n,p,0}^*$ as the  normal reference distribution of $F_{n,p,0}$. In what follows, we shall show that we can approximate the distribution of $F_{n,p,0}$ using the distribution of $F_{n,p,0}^*$ asymptotically.  For further study, we now derive the first three cumulants (mean, variance, and third central moment) of $T_{n,p,0}^*, S_{n,p,0}^*$, and $F_{n,p,0}^*$.
For simplicity, let $\calK_l(X), l=1,2,3$ denote the first three cumulants of a random variable $X$.
Then  the first three cumulants of $T_{n,p,0}^*$ are given by
\begin{equation}\label{EVA.equ}
	\calK_1(T_{n,p,0}^*)=\tr(\bOmega_n),\; \calK_2(T_{n,p,0}^*)=2\tr(\bOmega_n^2),\;\mbox{ and }\;\calK_3(T_{n,p,0}^*)=8\tr(\bOmega_n^3),
\end{equation}
see  \cite{ZhangTSBF2020} for details.  Furthermore, it follows from  Eq. (4) of \cite{zhang2005approximate} that
\begin{equation}\label{EVB.equ}
	\begin{split}
		\calK_1(S_{n,p,0}^*)&=\tr(\bOmega_n),\\ \calK_2(S_{n,p,0}^*)&=2\left[\frac{n_2^2}{n^2(n_1-1)}\tr(\bSigma_1^2)+\frac{n_1^2}{n^2(n_2-1)}\tr(\bSigma_2^2)\right],\;\mbox{ and }\\
		\calK_3(S_{n,p,0}^*)&=8\left[\frac{n_2^3}{n^3(n_1-1)^2}\tr(\bSigma_1^3)+\frac{n_1^3}{n^3(n_2-1)^2}\tr(\bSigma_2^3)\right].
	\end{split}
\end{equation}
By (\ref{EVA.equ}) and (\ref{EVB.equ}), the first three cumulants of $F_{n,p,0}^*$ are then  given by
\be\label{3cumuFstar.equ}
\begin{split}
	\calK_1(F_{n,p,0}^*)&=\frac{\calK_1(T_{n,p,0}^*)}{\calK_1(S_{n,p,0}^*)}[1+o(1)]=1+o(1),\\
	\calK_2(F_{n,p,0}^*)&=\left[\frac{\calK_2(T_{n,p,0}^*)}{\calK_1^2(S_{n,p,0}^*)}+\frac{\calK_1^2(T_{n,p,0}^*)}{\calK_1^4(S_{n,p,0}^*)}\calK_2(S_{n,p,0}^*)\right][1+o(1)]\\
	&=2\left[\frac{\tr(\bOmega_n^2)}{\tr^2(\bOmega_n)}+
	\frac{\frac{n_2^{2}}{n_1-1}\tr(\bSigma_1^2)+\frac{n_1^{2}}{n_2-1}\tr(\bSigma_2^2)}{n^2\tr^2(\bOmega_n)} \right][1+o(1)],\;\mbox{ and }\\
	\calK_3(F_{n,p,0}^*)&=\left[\frac{\calK_3(T_{n,p,0}^*)}{\calK_1^3(S_{n,p,0}^*)}+\frac{\calK_1^3(T_{n,p,0}^*)}{\calK_1^6(S_{n,p,0}^*)}\calK_3(S_{n,p,0}^*)\right][1+o(1)]\\
	&=8\left[\frac{\tr(\bOmega_n^3)}{\tr^3(\bOmega_n)}+\frac{\frac{n_2^3}{(n_1-1)^{2}}\tr(\bSigma_1^3)+\frac{n_1^3}{(n_2-1)^{2}}\tr(\bSigma_2^3)}{n^3\tr^3(\bOmega_n)} \right][1+o(1)],
	%&=8(d_1^{-3/2}+d_2^{-3/2})[1+o(1)],
\end{split}
\ee
by dropping the higher order terms.

Set $\rho_{n,p,r}=\lambda_{n,p,r}/\sqrt{\tr(\bOmega_n^2)}, r=1,\ldots,p$ which are the eigenvalues of $\bOmega_n/\sqrt{\tr(\bOmega_n^2)}$ in descending order. We impose the following conditions:
\begin{enumerate}
	\item [{C1.}] We assume  $\by_{ij}=\bmu_{i}+\bGamma_i\bz_{ij}, j=1,\dots,n_{i}, i=1,2$,
	where each $\bGamma_i$ is a $p\times m$ matrix for some $m\geq p$ such that  $\bGamma_i\bGamma_i^{\top}=\bSigma_i$
	and $\bz_{ij}$'s are i.i.d. $m$-vectors with $\E(\bz_{ij})=\0$
	and $\Cov(\bz_{ij})=\bI_{m}$, the $m\times m$ identity matrix.
	\item [{C2.}] Assume $\E(z_{ijk}^{4})=3+\Delta<\infty$ where $z_{ijk}$
	is the $k$-th component of $\bz_{ij}$, $\Delta$ is some constant,
	and $\E[z_{ij1}^{\alpha_{1}}\dots z_{ijp}^{\alpha_{p}}]=\E[z_{ijl_1}^{\alpha_1}]\ldots\E[z_{ijl_q}^{\alpha_q}]$ for a positive integer $q$ such that $\sum_{l=1}^q\alpha_l\leq 8$ and $l_1\neq \cdots \neq l_q$.
	\item[{C3.}] As $n\to\infty$, we have $n_1/n\to\tau\in(0,1)$.
	\item[{C4.}] We assume that  $\lim_{n,p\to\infty} \rho_{n,p,r}=\rho_{r},  r=1,2,\ldots$, uniformly  and $\lim_{n,p\to\infty} \sum_{r=1}^p \rho_{n,p,r}=\sum_{r=1}^{\infty} \rho_r<\infty$.
	\item[{C5.}] As $p\to\infty$, we have $\tr(\bSigma_i\bSigma_j\bSigma_l\bSigma_h)=o\left\{\tr^{2}[(\bSigma_1+\bSigma_2)^{2}]\right\}$ as $p\to\infty$  for all $i,j,l,h=1$ or $2$.
	%	\item[{C6.}] As $n,p\to\infty$, we have $p/n\to c\in (0,\infty)$.
\end{enumerate}
Conditions C1 and C2  are also  imposed by \cite{chen2010two} and \cite{bai1996effect}.  They specify a factor model for high-dimensional data analysis. Condition C3 is a regularity condition for two-sample problems. It ensures that the two sample sizes
$n_1$ and $n_2$ tend to infinity proportionally. Under Condition C3, by (\ref{3cumuFstar.equ}),  as $n\to\infty$, we have
\be\label{omega.equ}
\bOmega_n\to\bOmega=(1-\tau)\bSigma_1+\tau\bSigma_2,
\ee
and
\be\label{3cFstarlimit.equ}
\begin{split}
\calK_1(F_{n,p,0}^*)&=1+o(1),\;\calK_2(F_{n,p,0}^*)=\frac{2\tr(\bOmega_n^2)}{\tr^2(\bOmega_n)}[1+o(1)],\;\mbox{ and }\;\\
 \calK_3(F_{n,p,0}^*)&=\frac{8\tr(\bOmega_n^3)}{\tr^3(\bOmega_n)}[1+o(1)].
 \end{split}
\ee
Condition C4 ensures the existence of the limits of $\lambda_{n,p,r}$ as $n,p\to\infty$ and that the limit and summation operations in the expression $\lim_{n,p\to\infty}\sum_{r=1}^p\rho_{n,p,r},$ are exchangeable. It is used to ensure that the limiting distributions of the normalized  $F_{n,p,0}$ and $F_{n,p,0}^*$, namely,
% \be\label{tF.equ}
\[
\tF_{n,p,0}=\frac{F_{n,p,0}-1}{\sqrt{2\tr(\bOmega_n^2)/\tr^2(\bOmega_n)}},\;\mbox{ and }\;\tF_{n,p,0}^*=\frac{F_{n,p,0}^*-1}{\sqrt{2\tr(\bOmega_n^2)/\tr^2(\bOmega_n)}},
%\ee
\]
are not normal. Condition C5 is imposed by \cite{chen2010two} which is used to ensure that the limiting distributions of  $\tF_{n,p,0}$ and $\tF_{n,p,0}^*$ are  normal. %Condition C6 is imposed for studying the ratio-consistency of the estimators used in the proposed test.
Let  $\convL$ denote convergence in distribution, respectively.  We have the following useful theorem.
\begin{theorem}\label{limit.thm}
	(a) Under Conditions C1--C4, as $n,p\to\infty$, we have
	%\begin{equation}\label{limita.equ}
	\[
	\tF_{n,p,0}\convL \zeta,\;\mbox{ and }\;\tF^*_{n,p,0}\convL \zeta,
	\]
	%\end{equation}
	where $\zeta\dequ \sum_{r=1}^\infty \rho_{p,r}(A_r-1)/\sqrt{2}$.\\
	(b) Under Conditions C1--C3 and C5, as $n,p\to\infty$, we have
	%\begin{equation}\label{limitb.equ}
	\[
	\tF_{n,p,0}\convL \calN(0,1),\;\mbox{ and }\;\tF_{n,p,0}^*\convL \calN(0,1).
	%\end{equation}
	\]
	Then under the conditions of (a) or (b), we always have
	\begin{equation}\label{supcvg.sec2}
		\sup_{x}\vert\Pr(F_{n,p,0}\leq x)-\Pr(F_{n,p,0}^*\leq x )\vert\to 0.
	\end{equation}
\end{theorem}
\subsection{Null distribution approximation}

Theorem~\ref{limit.thm} shows that $F_{n,p,0}$ and $F_{n,p,0}^*$ have the same non-normal limit when Conditions C1--C4 hold, and $F_{n,p,0}$ and $F_{n,p,0}^*$ have the same normal limit when Condition C1--C3 and C5 hold. Therefore, it is theoretically justified  that we can  approximate the distribution of $F_{n,p,0}$ by that of $F_{n,p,0}^*$. Note that by  (\ref{Tnp0star.equ}),  (\ref{Snp0star.equ}), and (\ref{Fnp0star.equ}),  $F_{n,p,0}^*=T_{n,p,0}^*/S_{n,p,0}^*$  is a ratio of two independent $\chi^2$-type mixtures $T_{n,p,0}^*$ and $S_{n,p,0}^*$, known as an $F$-type mixture, as described in \cite[Sec. 4.4]{zhang2013analysis}. We can approximate the distribution of an $F$-type mixture via approximating the distributions of $T_{n,p,0}^*$ and $S_{n,p,0}^*$ respectively by  the Welch--Satterthwaite (W--S) $\chi^2$-approximation \cite{satterthwaite1946approximate,welch1947generalization,zhangguozhoucheng2020JASA}. For simplicity, we may call this approach as the $F$-approximation to the distribution of $F_{n,p,0}^*$, which proceeds  as follows.

For the $F$-type mixture $F_{n,p,0}^*=T_{n,p,0}^*/S_{n,p,0}^*$, we first  approximate the distributions of $T_{n,p,0}^*$ and $S_{n,p,0}^*$  using those of the following two independent scaled $\chi^2$-random variables:
%\be\label{R1R2.sec2}
\[
R_1\dequ \beta_1\chi_{d_1}^2, \;\mbox{ and }\;R_2\dequ \beta_2\chi_{d_2}^2,
\]
%\ee
and then approximate the distribution of $F_{n,p,0}^*$ by  that of $R_1/R_2$ where $\beta_1, d_1, \beta_2$, and $d_2$ are approximation parameters.  The approximation parameters $\beta_1$ and $d_1$  are determined via matching the first two cumulants of $T_{n,p,0}^*$  and $R_1$. The first two cumulants of $R_1$ are $\beta_1 d_1$ and $2\beta_1^2 d_1$ while the first two cumulants of $T_{n,p,0}^*$ are given in (\ref{EVA.equ}).  Equating the first two cumulants of $T_{n,p,0}^*$ and $R_1$ leads to
\be\label{betdf1.sec2}
\beta_1=\tr(\bOmega_n^2)/\tr(\bOmega_n), \;\mbox{ and }\; d_1= \tr^2(\bOmega_n)/\tr(\bOmega_n^2).
\ee
Similarly, the approximation  parameters $\beta_2$ and $d_2$  are determined via matching the first two cumulants of  $S_{n,p,0}^*$ and $R_2$. The first two cumulants of $R_2$ are $\beta_2 d_2$ and $2\beta_2^2 d_2$ while the first two cumulants  of $S_{n,p,0}^*$ are presented in (\ref{EVB.equ}). Equating the first two cumulants of $S_{n,p,0}^*$ and $R_2$ then leads to
\be\label{betdf2.sec2}
\begin{array}{rcl}
	\beta_2&=&[(n_2/n)^2(n_1-1)^{-1}\tr(\bSigma_1^2)+(n_1/n)^2(n_2-1)^{-1}\tr(\bSigma_2^2)]/\tr(\bOmega_n), \;\mbox{ and }\\
	d_2&=& \tr^2(\bOmega_n)/[(n_2/n)^2(n_1-1)^{-1}\tr(\bSigma_1^2)+(n_1/n)^2(n_2-1)^{-1}\tr(\bSigma_2^2)].
\end{array}
\ee
By the construction of our $F$-type test with (\ref{Fnp.equ}), (\ref{Snp.sec1}) and (\ref{trOmega.sec1}), we have $\E(T_{n,p,0}^*)=\E(S_{n,p,0}^*)$, implying  that $\beta_1 d_1=\beta_2 d_2$. Since $T_{n,p,0}^*$ and $S_{n,p,0}^*$ are independent, so are $R_1$ and $R_2$,  we have
\be\label{Fapprox.sec2}
\frac{R_1}{R_2}\dequ \frac{\beta_1 \chi_{d_1}^2}{\beta_2\chi_{d_2}^2}= \frac{\chi_{d_1}^2/d_1}{\chi_{d_2}^2/d_2}\sim F_{d_1,d_2},
\ee
%Note that
%\[
%\frac{\chi_{v_1}^2/v_1}{\chi_{v_2}^2/v_2}\sim F_{v_1,v_2},
%\]
where $F_{d_1,d_2}$ denotes the usual $F$ distribution with  $d_1$ and $d_2$ degrees of freedom. That is to say, we   approximate the  distribution of $F_{n,p,0}^*$ and hence the distribution of $F_{n,p,0}$  using $F_{d_1,d_2}$ where $d_1$ and $d_2$ are given in (\ref{betdf1.sec2}) and (\ref{betdf2.sec2}), respectively. Let $\hd_1$ and $\hd_2$ be the ratio-consistent estimators of $d_1$ and $d_2$. Then for any nominal significance level $\alpha>0$, the proposed $F$-type test can be conducted via using the critical value $F_{\hd_1,\hd_2}(\alpha)$ or the $p$-value $\Pr(F_{\hd_1,\hd_2}\geq F_{n,p})$ where  $F_{v_1,v_2}(\alpha)$ denotes the upper $100\alpha$ percentile of $F_{v_1,v_2}$.
\begin{remark}\label{FL2com.rem}  By (\ref{Fapprox.sec2}), it is clear that when $d_2\to\infty$, we have  $R_1/R_2\convL\chi_{d_1}^2/d_1$,  showing that when  $d_2$ is large,  our $F$-approximation to the distribution of $F_{n,p,0}^*$ is comparable with  the W--S $\chi^2$-approximation to the distribution
	of $T_{n,p,0}^*$ as  described  in \cite{ZhangTSBF2020}. Since $n_1, n_2<n$, by (\ref{betdf2.sec2}), we have
%	\be\label{d2bound.sec2}
\[
	d_2\geq \frac{n\tr^2(\bOmega_n)}{(n_2/n)^2\tr(\bSigma_1^2)+(n_1/n)^2\tr(\bSigma_2^2)}\geq n.
	\]
%	\ee
	It follows that as $n\to\infty$, we generally have $d_2\to\infty$. That is to say, for large sample cases, in terms of accuracy,  our $F$-type test is comparable with the $L^2$-norm based test of \cite{ZhangTSBF2020}. However, when $d_2$ is small, it is not the case   and in this case, our $F$-type test is expected to outperform  the $L^2$-norm based test of \cite{ZhangTSBF2020} since
	we take the variation of $T_{n,p}$ into account in the construction of our $F$-type test (\ref{Fnp.equ}). This is actually confirmed by the simulation results presented in Tables~\ref{size1.tab} and ~\ref{size2.tab} of Section~\ref{simu.sec}.
\end{remark}

\begin{remark}\label{Fapprox.rem} It is of interest to see if the $F$-approximation  $F_{d_1,d_2}$ can automatically mimic  the  distribution shape of the $F$-type mixture $F_{n,p,0}^*$.  %It is easy to find that $d_2>d_1$.
	Under Condition C3,  as $n\to\infty$, by (\ref{3cFstarlimit.equ}), the skewness of $F_{n,p,0}^*$ is  given by
	\be\label{skewFnp0.equ}
	\frac{\calK_3(F_{n,p,0}^*)}{\calK_2^{3/2}(F_{n,p,0}^*)}=(8/d^*)^{1/2}[1+o(1)], \mbox{ where } d^*=\frac{\tr^3(\bOmega_n^2)}{\tr^2(\bOmega_n^3)}.
	\ee
	By Theorem 5 of \cite{zhangguozhoucheng2020JASA} and Remark 1 of \cite{ZhangTSBF2020}, we have $1\leq d^*\leq d_1 \leq d_2$. By (\ref{skewFnp0.equ}), when $F_{n,p,0}^*$ is asymptotically normal, we have $d^*,d_1, d_2\to\infty$, and hence $R_1/R_2$ is also asymptotically normal. However,  when $d_1$ is asymptotically bounded, so is $d^*$, and hence both $F_{n,p,0}^*$ and $R_1/R_2$ will not be asymptotically normal. Therefore, this $F$-approximation is also adaptive to the underlying distribution shape of $F_{n,p,0}^*$.
\end{remark}

We now study to find the ratio-consistent estimators of the approximation degrees of freedom  $d_1$ and $d_2$. By (\ref{betdf1.sec2}) and (\ref{betdf2.sec2}), it is sufficient  to find the ratio-consistent estimators of  $\tr^2(\bOmega_n)$, $\tr(\bOmega_n^2)$,
and  $(n_2/n)^{2}(n_1-1)^{-1}\tr(\bSigma_1^2)+(n_1/n)^{2}(n_2-1)^{-1}\tr(\bSigma_2^2)$.
%For this end, we re-express $\tr^2(\bSigma_1/n_1+\bSigma_2/n_2)$ and $\tr\left(\bSigma_1/n_1+\bSigma_2/n_2\right)^2$ as
%\begin{equation}\label{reex.equ}
%\begin{array}{rcl}
%\tr^2(\bSigma_1/n_1+\bSigma_2/n_2)&=&\frac{1}{n_1^2}\tr^2(\bSigma_1)+\frac{2}{n_1n_2}\tr(\bSigma_1)\tr(\bSigma_2)+\frac{1}{n_2^2}\tr^2(\bSigma_2),\;\mbox{ and}\\
%\tr\left(\bSigma_1/n_1+\bSigma_2/n_2\right)^2&=&\frac{1}{n_1^2}\tr(\bSigma_1^2)+\frac{2}{n_1n_2}\tr(\bSigma_1\bSigma_2)+\frac{1}{n_2^2}\tr(\bSigma_2^2).\\
%\end{array}
%\end{equation}
Using the usual sample covariance matrices as defined in (\ref{hSigma.sec1}), following \cite{ZhangTSBF2020}, we estimate $\tr^2(\bSigma_i)$ and $\tr(\bSigma_i^2),i=1,2$ by
\[
%\begin{array}{rcl}
\begin{split}
	\widehat{\tr^2(\bSigma_i)}&=\frac{(n_i-1)n_i}{(n_i-2)(n_i+1)}\left[\tr^2(\hbSigma_i)-\frac{2}{n_i}\tr(\hbSigma_i^2)\right],\;\; \mbox{and}\\
	\widehat{\tr(\bSigma_i^2)}&=\frac{(n_i-1)^2}{(n_i-2)(n_i+1)}\left[\tr(\hbSigma_i^2)-\frac{1}{n_i-1}\tr^2(\hbSigma_i)\right],i=1,2.
\end{split}
%\end{array}
\]
The estimators of $d_1$ and $d_2$ are then given by
\be\label{hd12.equ}
	\hd_1=\widehat{\tr^2(\bOmega_n)}/\widehat{\tr(\bOmega_n^2)},\;\mbox{ and }\;
	\hd_2=\frac{\widehat{\tr^2(\bOmega_n)}}{\frac{n_2^2}{n^2(n_1-1)}\widehat{\tr(\bSigma_1^2)}+\frac{n_1^2}{n^2(n_2-1)}\widehat{\tr(\bSigma_2^2)}},
\ee
where
%\be\label{tr.equ}
\[
\begin{split}
	\widehat{\tr^2(\bOmega_n)}&=\frac{n_2^2}{n^2}\widehat{\tr^2(\bSigma_1)}+\frac{2n_1n_2}{n^2}\tr(\hbSigma_1)\tr(\hbSigma_2)+\frac{n_1^2}{n^2}\widehat{\tr^2(\bSigma_2)},\mbox{ and }\\
	\widehat{\tr(\bOmega_n^2)}&=\frac{n_2^2}{n^2}\widehat{\tr(\bSigma_1^2)}+\frac{2n_1n_2}{n^2}\tr(\hbSigma_1\hbSigma_2)+\frac{n_1^2}{n^2}\widehat{\tr(\bSigma_2^2)}.
\end{split}
\]
%\ee
Under Conditions C1--C3, as $n\to\infty$, we can show that  $\hd_1/d_1\convP 1$ and $\hd_2/d_2\convP 1$ uniformly for all $p$, where $\convP 1$ means convergence in probability. Hence $F_{\hd_1,\hd_2}/F_{d_1,d_2}\convP 1$ uniformly for all $p$. Theorem~\ref{limit.thm} still holds when $d_1$ and $d_2$ are replaced by  $\hd_1$ and $\hd_2$ as defined in (\ref{hd12.equ}).

\subsection{Asymptotic power}

Following  \cite{chen2010two},  we derive  the asymptotic  power of the proposed  $F$-type  test under the following local alternative:
\begin{equation}\label{H1.equ}
	(\bmu_1-\bmu_2)^\top\bSigma_i(\bmu_1-\bmu_2)=o[n^{-1}\tr(\bSigma_1+\bSigma_2)^2],\;  i=1,2, \mbox{ as } n,p\rightarrow\infty.
\end{equation}
This implies  that we have
\[
\begin{split}
&\quad\Var\left[(n_1n_2/n)^{1/2}(\barbx_1-\barbx_2)^\top(\bmu_1-\bmu_2)\right]\\
&=(\bmu_1-\bmu_2)^\top\bOmega_n(\bmu_1-\bmu_2)=o[\tr(\bOmega_n^2)], \;\mbox{ as }n,p\to\infty.
\end{split}
\]
Therefore, we can write $T_{n,p}=[T_{n,p,0}+\frac{n_1n_2}{n}\|\bmu_1-\bmu_2\|^2][1+o_p(1)]$. Thus,   $F_{n,p}=[F_{n,p,0}+\frac{n_1n_2}{n}\|\bmu_1-\bmu_2\|^2/\tr(\hbOmega_n)][1+o_p(1)]$ under the local alternative (\ref{H1.equ}).
\begin{theorem}\label{power.thm}
	Assume that as $n,p\to\infty$, $\hd_1$ and $\hd_2$ are ratio-consistent for $d_1$ and $d_2$. \\
	(a) Under Conditions C1--C4 and the local alternative (\ref{H1.equ}), as $n,p\to\infty$, we have
	\[
	\Pr\left[F_{n,p}\geq F_{\hd_1,\hd_2}(\alpha)\right]=\Pr\left\{\zeta\geq \frac{F_{d_1, d_2}(\alpha)-1}{\sqrt{2/d_1}}-\frac{n\tau(1-\tau)\|\bmu_1-\bmu_2\|^2}{\left[2\tr(\bOmega^2)\right]^{1/2}}\right\}[1+o(1)],
	\]
	where $\zeta$ is defined in Theorem~\ref{limit.thm}(a), $\bOmega$ is defined in (\ref{omega.equ}) and $\tau$ is given in Condition C3.\\
	(b) Under Conditions C1--C3, C5 and the local alternative (\ref{H1.equ}), as $n,p\to\infty$, we have
	\[
	\Pr\left[F_{n,p}\geq F_{\hd_1,\hd_2}(\alpha)\right]
	=\Phi\left\{-z_{\alpha}+\frac{n\tau(1-\tau)\|\bmu_1-\bmu_2\|^2}{\left[2\tr(\bOmega^2)\right]^{1/2}}\right\}[1+o(1)],
	\]
	where $z_\alpha$ denotes the upper $100\alpha$-percentile of $\calN(0,1)$.
\end{theorem}
\begin{remark}
	Under the conditions of Theorem~\ref{power.thm}(b), the asymptotic power of $F_{n,p}$ is the same as those  of the $L^2$-norm based tests proposed and studied by \cite{ZhangTSBF2020} and \cite{chen2010two}. However, this is not the case under the conditions of Theorem~\ref{power.thm}(a).
\end{remark}

\section{Simulation Studies}\label{simu.sec}

In this section, we  present two simulation studies to compare the numerical performance of  our  $F$-type test  $F_{n,p}$ (\ref{Fnp.equ})  against two existing tests for the two-sample Behrens--Fisher (BF) problem (\ref{H0.equ}) for  high-dimensional data, including  the $L^2$-norm based   tests  of \cite{chen2010two} and  \cite{ZhangTSBF2020},  denoted as $T_{CQ}$ and $T_{n,p}$,
respectively. They are defined in (\ref{CQstat}) and  (\ref{Tstat.sec1}) respectively. \cite{chen2010two}'s test approximated its  null distribution using the normal approximation while \cite{ZhangTSBF2020}'s test approximated its null distribution using the well-known W--S $\chi^2$-approximation.  Throughout this section, the nominal size $\alpha$ is set as $5\%$ and the number of simulation runs is $N=10,000$.

The simulation process may be briefly described as follows. In each simulation run, the two samples (\ref{2samp.equ}) are generated from the factor model under Condition C1,  and then  the test statistics and $p$-values of the three considered tests are computed. The null hypothesis is rejected  when the $p$-values are smaller than the nominal size $\alpha$ and the proportions of the number of rejections out of $N$ runs are the empirical sizes (under the null hypothesis) or the empirical powers (under the alternative hypothesis). To measure  the overall performance of a test in maintaining the nominal size, the value of average relative error (ARE)  defined in \cite{zhang2011tech} is used.  The ARE value of a test is defined as $\ARE=100M^{-1}\sum_{j=1}^M\vert\halpha_j-\alpha\vert/\alpha$, where $\halpha_j,j=1,\ldots,M$ denote the empirical sizes under $M$ simulation settings.  A test with a smaller ARE value has a better overall performance in terms of maintaining the nominal size.
\subsection{Simulation 1}
In this simulation, we generate the two independent samples  from the following factor model:
\[
\by_{ij} =\bmu_i+ \bSigma_i^{1/2}\bz_{ij},j=1,...,n_i; \;i=1,2,
\]
where $\bz_{ij}=(z_{ij1},\ldots,z_{ijp})^\top, j=1,\ldots, n_i; i=1,2$ are i.i.d.  random variables with\ $\E(\bz_{ij})=\0$ and $\Cov(\bz_{ij})=\bI_p$. Without loss of generality, we specify $\bmu_1=\0$. For the null hypothesis, we set $\bmu_1=\bmu_2=\0$. For power consideration, we set $\bmu_2=\bmu_1+\delta\bh $. Hence  the mean vector difference $\bmu_1-\bmu_2$ is controlled by  the tuning parameters $\delta$ and $\bh$. For simplicity, we set $\bh=\bu/\|\bu\|$ with $\bu=[1,\ldots,p]^\top$ for any given dimension $p$. Here $\bh$  specifies the direction of the mean difference $\bmu_1-\bmu_2$. The tuning parameter $\delta$ controls the amount of the mean vector difference. It is clear that
when $\delta$ increases, the powers of the tests under consideration  are
expected to increase, and we shall let $\delta$ take
values in $[0,\delta_0]$ with $\delta_0$ properly chosen so that the
associated powers of the tests are properly scattered over  $[0,1]$.  The covariance matrices $\bSigma_1$ and $\bSigma_2$ are specified as $\bSigma_i=\sigma^2_i\left[(1-\rho_i)\bI_p+\rho_i\bJ_p\right],\; i=1,2$, where $\bJ_p$ is the $p\times p$ matrix of ones. Note that the covariance matrix difference $\bSigma_1-\bSigma_2$ is controlled by  the tuning parameters  $\sigma^2_i$ and $\rho_i,i=1,2$.  In particular,  $\sigma^2_1$ and  $\sigma^2_2$ control  the variance of the data, and we set $\sigma_1^2=1$ and $\sigma_2^2=2$ for simplicity. The tuning parameters  $\rho_1$ and $\rho_2$ control the correlation of the data, and we set  $\rho_1=0.1$ and consider three cases of $\rho_2$  as  $\rho_2=0.1,0.5$, and $0.9$. To compare  the performance of the considered tests on the data with different distribution shapes, we consider  to independently generate the $p$-vectors $\bz_{ij}$ using  the following three models: (1) $\bz_{ij}\iidsim \calN(0,1)$; (2)  $\bz_{ij}\iidsim t_4/\sqrt{2}$; (3) $\bz_{ij}\iidsim (\chi_1^2-1)/\sqrt{2}$,   representing  three distribution shapes: (1) normal; (2) non-normal but symmetric; (3) non-normal and skewed. Lastly, for the sample sizes $\bn=[n_1,n_2]$ and dimension $p$, we consider three cases of  $p=50,500,1000$ and three cases of   $\bn$ as  $\bn_1=[30,50], \bn_2=[120,200]$ and $\bn_3=[240,400]$. The numerical results on these three considered tests in Simulation 1 are summarized in Tables~\ref{size1.tab}$\sim$\ref{power1.tab}.
\begin{table}[!h]
	\begin{center}
		\begin{minipage}{0.8\textwidth}
			\caption{Simulation 1: Empirical sizes (in $\%$).}  \label{size1.tab}	
			\begin{tabular*}{\textwidth}{@{\extracolsep{\fill}}p{0.02\textwidth}p{0.01\textwidth}p{0.01\textwidth}p{0.01\textwidth}p{0.01\textwidth}p{0.01\textwidth}p{0.01\textwidth}p{0.01\textwidth}p{0.01\textwidth}p{0.01\textwidth}p{0.01\textwidth}p{0.01\textwidth}p{0.01\textwidth}@{\extracolsep{\fill}}}
				\hline
				&       &       & \multicolumn{3}{c}{$\rho_2=0.1$} & \multicolumn{3}{c}{$\rho_2=0.5$} & \multicolumn{3}{c}{$\rho_2=0.9$} \\
			\hline
				Model & $p$     & $\bn$ & $T_{CQ}$    & $T_{n,p}$   & $F_{n,p}$ & $T_{CQ}$    & $T_{n,p}$   & $F_{n,p}$ & $T_{CQ}$    & $T_{n,p}$   & $F_{n,p}$ \\
			\hline
				\multirow{9}[0]{*}{1} & \multirow{3}[0]{*}{$50$} & $\bn_1$ & 6.63  & 5.65  & 5.47  & 7.10  & 5.91  & 5.71  & 7.11  & 5.78  & 5.36 \\
				&       & $\bn_2$ & 6.37  & 5.44  & 5.35  & 6.67  & 5.48  & 5.42  & 6.80  & 5.26  & 5.15 \\
				&       & $\bn_3$ & 5.99  & 5.27  & 5.27  & 6.61  & 5.39  & 5.32  & 6.93  & 5.28  & 5.26 \\
				%	\cline{2-12}	
				& \multirow{3}[0]{*}{$500$} & $\bn_1$ & 7.17  & 6.80  & 6.73  & 6.99  & 6.11  & 6.01  & 7.11  & 5.72  & 5.36 \\
				&       & $\bn_2$ & 6.37  & 5.89  & 5.87  & 7.04  & 6.03  & 5.97  & 7.31  & 5.52  & 5.39 \\
				&       & $\bn_3$ & 6.46  & 6.01  & 6.00  & 6.16  & 5.22  & 5.22  & 7.09  & 5.43  & 5.38 \\
				%	\cline{2-12}
				& \multirow{3}[0]{*}{$1000$} & $\bn_1$ & 7.43  & 7.09  & 6.98  & 7.32  & 6.32  & 6.06  & 7.16  & 5.80  & 5.39 \\
				&       & $\bn_2$ & 6.60  & 6.17  & 6.13  & 6.81  & 5.78  & 5.70  & 6.77  & 5.46  & 5.40 \\
				&       &$\bn_3$ & 6.35  & 5.87  & 5.87  & 7.05  & 5.87  & 5.84  & 7.35  & 5.51  & 5.45 \\
			\hline
				\multirow{9}[0]{*}{2} & \multirow{3}[0]{*}{$50$} & $\bn_1$ & 6.15  & 4.71  & 4.54  & 7.01  & 5.93  & 5.73  & 6.94  & 5.68  & 5.33 \\
				&       & $\bn_2$  & 6.35  & 5.16  & 5.12  & 6.68  & 5.39  & 5.35  & 7.00  & 5.35  & 5.25 \\
				&       & $\bn_3$ & 6.07  & 5.19  & 5.16  & 6.97  & 5.71  & 5.68  & 7.06  & 5.53  & 5.48 \\
				%	\cline{2-12}
				& \multirow{3}[0]{*}{$500$} & $\bn_1$ & 6.27  & 5.80  & 5.74  & 7.44  & 6.40  & 6.16  & 7.12  & 5.80  & 5.46 \\
				&       &$\bn_2$ & 6.71  & 6.23  & 6.21  & 7.28  & 6.11  & 6.03  & 7.51  & 5.68  & 5.58 \\
				&       & $\bn_3$  & 6.37  & 5.85  & 5.84  & 6.92  & 5.83  & 5.78  & 6.74  & 5.39  & 5.34 \\
				%	\cline{2-12}
				& \multirow{3}[0]{*}{$1000$} & $\bn_1$ & 7.28  & 6.75  & 6.64  & 7.10  & 6.02  & 5.86  & 6.92  & 5.61  & 5.12 \\
				&       &$\bn_2$ & 7.03  & 6.51  & 6.48  & 7.07  & 5.89  & 5.80  & 7.58  & 5.94  & 5.86 \\
				&       & $\bn_3$ & 6.40  & 5.91  & 5.91  & 6.75  & 5.66  & 5.63  & 6.88  & 5.35  & 5.29 \\
			\hline
				\multirow{9}[0]{*}{3} & \multirow{3}[0]{*}{$50$} & $\bn_1$& 6.57  & 5.07  & 4.84  & 7.03  & 5.73  & 5.49  & 6.84  & 5.48  & 5.06 \\
				&       & $\bn_2$ & 6.37  & 5.20  & 5.17  & 6.53  & 5.08  & 5.02  & 6.70  & 5.17  & 5.09 \\
				&       &$\bn_3$& 6.34  & 5.39  & 5.36  & 6.82  & 5.65  & 5.60  & 7.25  & 5.49  & 5.47 \\
				%	\cline{2-12}
				& \multirow{3}[0]{*}{$500$} & $\bn_1$ & 6.77  & 6.13  & 6.02  & 6.95  & 6.11  & 5.97  & 7.36  & 6.00  & 5.58 \\
				&       &$\bn_2$ & 6.60  & 6.12  & 6.09  & 6.86  & 5.65  & 5.61  & 7.37  & 5.70  & 5.61 \\
				&       & $\bn_3$ & 6.69  & 6.24  & 6.23  & 6.84  & 5.72  & 5.71  & 6.93  & 5.17  & 5.14 \\
				%	\cline{2-12}
				& \multirow{3}[0]{*}{$1000$} & $\bn_1$ & 7.01  & 6.71  & 6.63  & 7.64  & 6.77  & 6.52  & 7.07  & 5.82  & 5.56 \\
				&       & $\bn_2$ & 6.72  & 6.16  & 6.13  & 6.93  & 5.74  & 5.67  & 6.88  & 5.44  & 5.34 \\
				&       & $\bn_3$ & 6.67  & 6.24  & 6.22  & 6.97  & 5.94  & 5.94  & 6.89  & 5.32  & 5.28 \\
			\hline
				\multicolumn{3}{c}{ARE} & 31.66 & 18.62 & 17.96 & 38.92 & 16.62 & 14.67 & 41.24 & 10.87 & 7.39 \\
			\hline
			\end{tabular*}
		\end{minipage}
	\end{center}
\end{table}

\begin{table}[!h]
	\begin{center}
		\begin{minipage}{0.8\textwidth}
			\caption{Simulation 1: Values of $\hd_1$ and $\hd_2$.}  \label{para1.tab}	
			\begin{tabular*}{\textwidth}{@{\extracolsep{\fill}}p{0.02\textwidth}p{0.01\textwidth}p{0.01\textwidth}p{0.01\textwidth}p{0.01\textwidth}p{0.01\textwidth}p{0.01\textwidth}p{0.01\textwidth}p{0.01\textwidth}@{\extracolsep{\fill}}}
				\hline
				&       &       &        \multicolumn{2}{c}{$\rho_2=0.1$} & \multicolumn{2}{c}{$\rho_2=0.5$} & \multicolumn{2}{c}{$\rho_2=0.9$} \\
				\hline
		Model	& $p$     & $\bn$ & $\hd_1$    & $\hd_2$    & $\hd_1$    & $\hd_2$   & $\hd_1$    & $\hd_2$     \\
			\hline
					\multirow{9}[0]{*}{1} & \multirow{3}[0]{*}{$50$} & $\bn_1$ & 34    & 2590  & 9     & 595   & 4     & 208 \\
					&       & $\bn_2$ & 34    & 10428 & 9     & 2279  & 3     & 804 \\
					&       & $\bn_3$ & 34    & 20885 & 8     & 4521  & 3     & 1594 \\
				%	\cline{2-9}
					& \multirow{3}[0]{*}{$500$} & $\bn_1$ & 90    & 6830  & 11    & 689   & 4     & 216 \\
					&       & $\bn_2$ & 85    & 26306 & 10    & 2577  & 4     & 828 \\
					&       & $\bn_3$ & 84    & 52350 & 10    & 5123  & 3     & 1642 \\
					%		\cline{2-9}
					& \multirow{3}[0]{*}{$1000$} & $\bn_1$ & 99    & 7534  & 11    & 693   & 4     & 218 \\
					&       & $\bn_2$ & 93    & 28773 & 10    & 2611  & 4     & 830 \\
					&       & $\bn_3$ & 92    & 57116 & 10    & 5152  & 3     & 1644 \\
					\hline
					\multirow{9}[0]{*}{2} & \multirow{3}[0]{*}{$50$} & $\bn_1$ & 31    & 2230  & 9     & 571   & 4     & 207 \\
					&       & $\bn_2$ & 33    & 9851  & 8     & 2248  & 3     & 800 \\
					&       & $\bn_3$ & 33    & 20217 & 8     & 4490  & 3     & 1592 \\
						%	\cline{2-9}
					& \multirow{3}[0]{*}{$500$} & $\bn_1$ & 87    & 6382  & 10    & 680   & 4     & 216 \\
					&       & $\bn_2$ & 84    & 25767 & 10    & 2581  & 4     & 828 \\
					&       & $\bn_3$ & 84    & 51719 & 10    & 5113  & 3     & 1642 \\
						%	\cline{2-9}
					& \multirow{3}[0]{*}{$1000$} & $\bn_1$ & 96    & 7207  & 11    & 688   & 4     & 217 \\
					&       & $\bn_2$ & 92    & 28449 & 10    & 2602  & 4     & 829 \\
					&       & $\bn_3$ & 92    & 56713 & 10    & 5162  & 3     & 1646 \\
					\hline
					\multirow{9}[0]{*}{3} & \multirow{3}[0]{*}{$50$} & $\bn_1$ & 31    & 2177  & 9     & 572   & 3     & 206 \\
					&       & $\bn_2$ & 33    & 9907  & 8     & 2260  & 3     & 801 \\
					&       & $\bn_3$ & 33    & 20341 & 8     & 4502  & 3     & 1593 \\
						%\cline{2-9}
					& \multirow{3}[0]{*}{$500$} & $\bn_1$ & 87    & 6479  & 11    & 683   & 4     & 217 \\
					&       & $\bn_2$ & 85    & 26031 & 10    & 2580  & 4     & 829 \\
					&       & $\bn_3$ & 84    & 52043 & 10    & 5121  & 3     & 1642 \\
						%	\cline{2-9}
					& \multirow{3}[0]{*}{$1000$} & $\bn_1$ & 97    & 7321  & 11    & 695   & 4     & 216 \\
					&       & $\bn_2$ & 93    & 28596 & 10    & 2602  & 4     & 829 \\
					&       & $\bn_3$ & 92    & 57015 & 10    & 5158  & 3     & 1645 \\
			\hline
		\end{tabular*}
	\end{minipage}
\end{center}
\end{table}
Table~\ref{size1.tab} displays the empirical sizes of $T_{CQ}$, $T_{n,p}$, and $F_{n,p}$ with the last row being their ARE values associated with the three values of $\rho_2$. We can draw several interesting conclusions in terms of size control. Firstly, the empirical sizes of $F_{n,p}$ are generally around $5\%$ under Models 1, 2, and 3, showing that the proposed $F$-type test  also  works well for the simulated  non-normal data.  Secondly,  $F_{n,p}$ always outperforms  $T_{n,p}$ since  for $\rho_2=0.1,0.5$, and $0.9$,  the ARE values of $F_{n,p}$ are $17.96$, $14.67$, and $7.39$, respectively   while  these values for $T_{n,p}$ are $18.62$, $16.62$, and $10.87$, respectively. This result is also confirmed  by the empirical size ranges of $F_{n,p}$ and $T_{n,p}$.  The empirical sizes of $F_{n,p}$ range from $4.54\%$ to $6.98\%$ while the empirical sizes of $T_{n,p}$  range from $4.71\%$ to $7.09\%$. Thirdly, with increasing the total sample size $n=n_1+n_2$,  the performances of  $F_{n,p}$ and $T_{n,p}$ become closer and closer. This is consistent with Remark~\ref{FL2com.rem} presented in Section~\ref{main.sec}.  Lastly, $T_{n,p}$ and $F_{n,p}$ generally outperform $T_{CQ}$  since  $T_{CQ}$ always has  the largest ARE values under the three settings associated with  $\rho_2=0.1,0.5$, and $0.9$.  Therefore,
in terms of size control, our $F$-type test $F_{n,p}$  indeed  improves the $L^2$-norm based tests $T_{CQ}$ and $T_{n,p}$ proposed by \cite{chen2010two} and  \cite{ZhangTSBF2020}, respectively.

In order to partially explain,  in terms of size control,  why $T_{CQ}$  does not perform well,  why $F_{n,p}$ and $T_{n,p}$ outperform $T_{CQ}$ generally,   and why $F_{n,p}$ always   outperforms  $T_{n,p}$ but with increasing the total sample size $n$, their performances are getting closer and closer, the values of $\hd_1$ and  $\hd_2$  are presented in Table~\ref{para1.tab}. Firstly, it is seen that the values of  $\hd_1$ are  generally not very large when $\rho_2=0.1$ and they are quite small when $\rho_2=0.5$ and $0.9$. This means that the normal approximation to the null distribution of $T_{CQ}$ will be less adequate or not adequate at all. However, the W--S $\chi^2$-approximation and the  $F$-approximation to the null distribution of $T_{n,p}$ and $F_{n,p}$ are not affected too much. This partially  explains  why $T_{CQ}$  does not perform well when $\rho_2=0.5$ and $0.9$ and  why $F_{n,p}$ and $T_{n,p}$ outperform $T_{CQ}$ generally. Secondly, it is seen that the values of $\hd_2$ are generally larger than the total sample size $n$ and are generally quite large. This partially explains why $T_{n,p}$ still performs quite well under each setting. Further,  with increasing the total sample size $n$,  the values of $\hd_2$ are also increasing, so that, as seen from Table~\ref{size1.tab}, in terms of size control, the performances of $T_{n,p}$ and $F_{n,p}$ are getting closer and closer. These results are consistent with the conclusions drawn in Remark~\ref{Fapprox.rem}.

\begin{table}[!h]
	\begin{center}
		\begin{minipage}{0.8\textwidth}
			\caption{Simulation 1: Empirical powers (in $\%$).}  \label{power1.tab}	
			\begin{tabular*}{\textwidth}{@{\extracolsep{\fill}}p{0.6cm}p{0.4cm}p{0.5cm}p{0.5cm}p{0.5cm}p{0.5cm}p{0.5cm}p{0.5cm}p{0.5cm}p{0.5cm}p{0.5cm}p{0.5cm}p{0.5cm}}

			\hline
			\multicolumn{4}{c}{}& \multicolumn{3}{c}{$\rho_2=0.1$} & \multicolumn{3}{c}{$\rho_2=0.5$} & \multicolumn{3}{c}{$\rho_2=0.9$} \\
			\hline
			Model	& $p$     & $\bn$ &$\delta$ & $T_{CQ}$    & $T_{n,p}$   & $F_{n,p}$ & $T_{CQ}$    & $T_{n,p}$   & $F_{n,p}$ & $T_{CQ}$    & $T_{n,p}$   & $F_{n,p}$ \\
			\hline
				\multirow{9}[0]{*}{1} & \multirow{3}[0]{*}{$50$} & $\bn_1$ & 1.5  & 59.63 & 57.25 & 56.88 & 31.01 & 28.18 & 27.49 & 21.35 & 18.66 & 17.72 \\
				&       & $\bn_2$ & 0.7  & 52.95 & 50.46 & 50.31 & 26.84 & 24.02 & 23.91 & 18.77 & 15.87 & 15.66 \\
				&       & $\bn_3$ & 0.5  & 54.16 & 51.36 & 51.31 & 27.40 & 24.61 & 24.52 & 19.15 & 15.97 & 15.85 \\
			%	\cline{2-13}
				& \multirow{3}[0]{*}{$500$} & $\bn_1$& 4.0  & 59.58 & 58.35 & 58.07 & 24.38 & 22.12 & 21.54 & 17.57 & 14.98 & 14.30 \\
				&       & $\bn_2$ & 2.0  & 59.17 & 57.81 & 57.76 & 24.36 & 21.79 & 21.72 & 16.85 & 13.99 & 13.80 \\
				&       & $\bn_3$& 1.3  & 51.97 & 50.31 & 50.28 & 21.42 & 18.95 & 18.89 & 15.04 & 12.31 & 12.21 \\
			%	\cline{2-13}
				& \multirow{3}[0]{*}{$1000$} & $\bn_1$& 5.4 & 56.59 & 55.34 & 54.99 & 23.43 & 21.23 & 20.77 & 16.49 & 13.92 & 13.16 \\
				&       & $\bn_2$ & 2.5  & 50.24 & 49.06 & 48.98 & 20.15 & 17.92 & 17.79 & 14.47 & 11.87 & 11.67 \\
				&       & $\bn_3$ & 1.9  & 55.64 & 54.09 & 54.05 & 22.41 & 19.93 & 19.87 & 15.46 & 12.70 & 12.61 \\
			\hline
				\multirow{9}[0]{*}{2} & \multirow{3}[0]{*}{$50$} & $\bn_1$ & 1.5  & 59.48 & 55.23 & 54.58 & 30.75 & 27.54 & 26.83 & 21.41 & 18.33 & 17.43 \\
				&       &$\bn_2$ & 0.7  & 53.23 & 49.82 & 49.71 & 27.21 & 24.28 & 24.10 & 18.56 & 15.67 & 15.42 \\
				&       & $\bn_3$ & 0.5  & 54.15 & 51.29 & 51.26 & 27.09 & 23.88 & 23.84 & 19.93 & 16.34 & 16.23 \\
			%	\cline{2-13}
				& \multirow{3}[0]{*}{$500$} & $\bn_1$ & 4.0  & 61.07 & 59.13 & 58.78 & 24.75 & 22.29 & 21.68 & 17.24 & 14.39 & 13.67 \\
				&       & $\bn_2$ & 2.0  & 59.68 & 57.92 & 57.86 & 23.86 & 21.46 & 21.31 & 17.57 & 14.83 & 14.61 \\
				&       & $\bn_3$ & 1.3  & 52.16 & 50.52 & 50.47 & 20.55 & 18.48 & 18.40 & 15.19 & 12.39 & 12.30 \\
			%	\cline{2-13}
				& \multirow{3}[0]{*}{$1000$} & $\bn_1$ & 5.4  & 56.84 & 55.35 & 55.00 & 23.33 & 21.38 & 21.06 & 17.14 & 14.55 & 13.93 \\
				&       & $\bn_2$ & 2.5  & 49.48 & 48.06 & 47.96 & 20.45 & 18.16 & 18.02 & 14.18 & 11.86 & 11.75 \\
				&       & $\bn_3$ & 1.9  & 56.65 & 55.10 & 55.05 & 22.86 & 20.35 & 20.27 & 15.41 & 12.88 & 12.79 \\
					\hline
				\multirow{9}[0]{*}{3} & \multirow{3}[0]{*}{$50$} & $\bn_1$& 1.5  & 58.75 & 55.00 & 54.48 & 29.98 & 26.36 & 25.68 & 20.83 & 18.02 & 17.20 \\
				&       &$\bn_2$ & 0.7  & 52.44 & 49.74 & 49.62 & 26.34 & 23.56 & 23.40 & 18.96 & 15.60 & 15.41 \\
				&       & $\bn_3$ & 0.5  & 53.73 & 50.70 & 50.64 & 26.65 & 23.96 & 23.87 & 19.30 & 15.87 & 15.82 \\
			%	\cline{2-13}
				& \multirow{3}[0]{*}{$500$} & $\bn_1$ & 4.0  & 59.59 & 57.99 & 57.59 & 24.76 & 22.57 & 22.16 & 16.87 & 14.38 & 13.71 \\
				&       & $\bn_2$ & 2.0  & 59.79 & 58.28 & 58.19 & 23.57 & 20.78 & 20.59 & 16.88 & 14.15 & 13.84 \\
				&       & $\bn_3$ & 1.3  & 52.73 & 51.15 & 51.11 & 21.17 & 18.74 & 18.68 & 14.85 & 12.09 & 11.98 \\
			%	\cline{2-13}
				& \multirow{3}[0]{*}{$1000$} & $\bn_1$& 5.4  & 57.13 & 55.68 & 55.41 & 22.67 & 20.58 & 20.17 & 16.36 & 14.07 & 13.35 \\
				&       & $\bn_2$ & 2.5  & 50.06 & 48.86 & 48.78 & 19.25 & 17.26 & 17.15 & 14.59 & 12.01 & 11.90 \\
				&       & $\bn_3$ & 1.9  & 56.42 & 54.99 & 54.97 & 23.00 & 20.54 & 20.51 & 16.06 & 13.28 & 13.15 \\
				\hline
			\end{tabular*}
		\end{minipage}
	\end{center}
\end{table}

We now investigate the empirical powers of $T_{CQ}$, $T_{n,p}$, and $F_{n,p}$. Table~\ref{power1.tab} presents their empirical powers obtained in Simulation 1. It is seen that $T_{n,p}$ and $F_{n,p}$ have comparable empirical powers under each setting since, as seen from Table~\ref{size1.tab}, their empirical sizes are generally comparable under each setting. It is also seen that $T_{CQ}$ has slightly higher powers than $T_{n,p}$ and $F_{n,p}$. This is also reasonable since,  as seen from  Table~\ref{size1.tab}, the empirical sizes of $T_{CQ}$ are  generally  liberal and they are generally larger than those of $T_{n,p}$ and $F_{n,p}$.

\subsection{Simulation 2}

In this simulation, we continue to use the setup of Simulation 1 except we now set $\bSigma_i=\bD\bR_i\bD,i=1,2$, where $\bD=\diag(d_1,\ldots,d_p)$ with $d_k=(p-k+1)/p,k=1,\ldots,p$, and $\bR_i=(r_{i,k\ell})$ with $r_{i,k\ell}=(-1)^{k+\ell}\rho_i^{0.1\vert k-\ell\vert},k,\ell=1,\ldots,p;i=1,2.$ In this case, the diagonal elements of $\bSigma_i$ are different and the larger difference between $\rho_1$ and $\rho_2$ will mean  the larger difference between $\bSigma_1$ and $\bSigma_2$.
Note that the tuning parameters $\rho_1$ and $\rho_2$  here play a somewhat different role from that in Simulation 1. In Simulation 1, $\rho_2=0.1,0.5$, and $0.9$ mean that  the simulated data are nearly uncorrelated, moderately correlated, and highly correlated, respectively. However, in this simulation, even when $\rho_2=0.9$, the correlation of the simulated data may still  be very small especially when $p$ is large. Nevertheless, the  values of $\rho_1$ and $\rho_2$ are also strongly related to the correlation of the simulated data.  The larger the value of $\rho_i$ is, the larger correlation among the $p$ variables of the simulated data is. The absolute correlation value of $r_{i,k\ell},i=1,2$ decays as $\vert k-\ell\vert$ increases.
\begin{table}[!h]
	\begin{center}
		\begin{minipage}{0.8\textwidth}
			\caption{Simulation 2: Empirical sizes (in $\%$).}  \label{size2.tab}	
			\begin{tabular*}{\textwidth}{@{\extracolsep{\fill}}p{0.6cm}p{0.4cm}p{0.5cm}p{0.5cm}p{0.5cm}p{0.5cm}p{0.5cm}p{0.5cm}p{0.5cm}p{0.5cm}p{0.5cm}p{0.5cm}}
				\hline
				&       &       & \multicolumn{3}{c}{$\rho_2=0.1$} & \multicolumn{3}{c}{$\rho_2=0.5$} & \multicolumn{3}{c}{$\rho_2=0.9$} \\
				\hline
			Model	& $p$     & $\bn$ & $T_{CQ}$    & $T_{n,p}$   & $F_{n,p}$ & $T_{CQ}$    & $T_{n,p}$   & $F_{n,p}$ & $T_{CQ}$    & $T_{n,p}$   & $F_{n,p}$ \\
				\hline
				\multirow{9}[0]{*}{1} & \multirow{3}[0]{*}{50} & $\bn_1$ & 7.41  & 5.80  & 5.57  & 7.50  & 6.00  & 5.61  & 6.91  & 5.52  & 5.16 \\
				&       & $\bn_2$ & 7.14  & 5.55  & 5.47  & 6.81  & 5.16  & 5.10  & 6.87  & 5.27  & 5.10 \\
				&       & $\bn_3$ & 6.68  & 5.34  & 5.31  & 6.94  & 5.18  & 5.11  & 6.99  & 5.35  & 5.33 \\
			%	\cline{2-12}
				& \multirow{3}[0]{*}{500} & $\bn_1$ & 6.24  & 5.54  & 5.33  & 6.81  & 5.87  & 5.58  & 6.94  & 5.92  & 5.72 \\
				&       & $\bn_2$ & 6.45  & 5.60  & 5.55  & 6.61  & 5.81  & 5.76  & 6.87  & 5.81  & 5.74 \\
				&       & $\bn_3$& 6.29  & 5.50  & 5.48  & 6.41  & 5.55  & 5.51  & 7.08  & 6.09  & 6.08 \\
			%	\cline{2-12}
				& \multirow{3}[0]{*}{1000} & $\bn_1$ & 6.00  & 5.45  & 5.25  & 6.11  & 5.61  & 5.42  & 6.86  & 6.15  & 5.93 \\
				&       &$\bn_2$ & 6.26  & 5.59  & 5.54  & 5.98  & 5.16  & 5.09  & 6.57  & 5.76  & 5.73 \\
				&       & $\bn_3$ & 6.25  & 5.76  & 5.74  & 6.78  & 6.00  & 5.99  & 6.47  & 5.64  & 5.64 \\
					\hline
				\multirow{9}[0]{*}{2} & \multirow{3}[0]{*}{50} & $\bn_1$ & 7.32  & 5.40  & 5.07  & 6.91  & 5.50  & 5.25  & 6.73  & 5.37  & 5.03 \\
				&       & $\bn_2$ & 6.77  & 5.08  & 4.99  & 7.23  & 5.42  & 5.35  & 6.50  & 4.93  & 4.89 \\
				&       & $\bn_3$ & 6.78  & 5.27  & 5.23  & 7.22  & 5.57  & 5.54  & 6.77  & 5.11  & 5.09 \\
			%	\cline{2-12}
				& \multirow{3}[0]{*}{500} & $\bn_1$ & 6.45  & 5.45  & 5.24  & 6.32  & 5.26  & 5.08  & 7.52  & 6.53  & 6.30 \\
				&       & $\bn_2$ & 6.18  & 5.20  & 5.15  & 6.52  & 5.66  & 5.61  & 6.49  & 5.31  & 5.27 \\
				&       &$\bn_3$ & 6.28  & 5.31  & 5.29  & 6.48  & 5.65  & 5.64  & 6.47  & 5.36  & 5.33 \\
			%	\cline{2-12}
				& \multirow{3}[0]{*}{1000} & $\bn_1$ & 5.91  & 5.19  & 5.04  & 6.25  & 5.25  & 5.11  & 6.84  & 6.11  & 5.92 \\
				&       &$\bn_2$ & 5.83  & 5.10  & 5.03  & 6.78  & 6.12  & 6.06  & 6.58  & 5.64  & 5.58 \\
				&       & $\bn_3$ & 5.84  & 5.23  & 5.20  & 6.08  & 5.19  & 5.18  & 6.62  & 5.76  & 5.73 \\
				\hline
				\multirow{9}[0]{*}{3} & \multirow{3}[0]{*}{50} & $\bn_1$ & 7.19  & 5.63  & 5.34  & 6.54  & 4.87  & 4.60  & 7.13  & 5.47  & 5.06 \\
				&       & $\bn_2$ & 7.35  & 5.71  & 5.66  & 6.87  & 5.17  & 5.11  & 6.75  & 5.24  & 5.19 \\
				&       & $\bn_3$ & 6.94  & 5.26  & 5.21  & 6.71  & 5.08  & 5.05  & 6.54  & 5.09  & 5.05 \\
			%	\cline{2-12}
				& \multirow{3}[0]{*}{500} & $\bn_1$ & 6.36  & 5.55  & 5.41  & 6.35  & 5.32  & 5.08  & 6.98  & 5.92  & 5.68 \\
				&       & $\bn_2$ & 6.67  & 5.71  & 5.67  & 6.88  & 5.67  & 5.64  & 6.90  & 5.87  & 5.77 \\
				&       & $\bn_3$ & 6.51  & 5.65  & 5.64  & 6.64  & 5.52  & 5.49  & 6.72  & 5.65  & 5.62 \\
			%	\cline{2-12}
				& \multirow{3}[0]{*}{1000} & $\bn_1$ & 5.92  & 5.21  & 5.05  & 6.72  & 5.75  & 5.64  & 6.70  & 5.81  & 5.62 \\
				&       & $\bn_2$ & 6.22  & 5.60  & 5.54  & 5.82  & 5.12  & 5.03  & 6.51  & 5.62  & 5.57 \\
				&       & $\bn_3$ & 6.11  & 5.43  & 5.42  & 6.30  & 5.51  & 5.50  & 6.62  & 5.68  & 5.64 \\
					\hline
				\multicolumn{3}{c}{ARE} & 20.89 & 8.97 & 6.99  & 32.27 & 9.80 & 8.10  & 35.50 & 12.68 & 10.36 \\
			\hline
			\end{tabular*}
		\end{minipage}
	\end{center}
\end{table}

The empirical sizes of $T_{CQ}$, $T_{n,p}$, and $F_{n,p}$  in Simulation 2 are presented  in Table~\ref{size2.tab}. The last row also displays their ARE values associated with the three values of $\rho_2$. We have similar conclusions to those drawn from Table~\ref{size1.tab}. In terms of size control, $F_{n,p}$  also  works well for the simulated  non-normal data;  $F_{n,p}$ always outperforms  $T_{n,p}$ and with increasing the total sample size, the performances of $F_{n,p}$ and $T_{n,p}$ are getting closer and closer; and both $F_{n,p}$ and $T_{n,p}$  generally outperform $T_{CQ}$.    Table~\ref{para2.tab} lists the values of $\hd_1$ and $\hd_2$ under various configurations in Simulation 2. Firstly, it is again seen that  $\hd_1$ are generally not very large even when $p=500$ and $1000$. This means that the normal approximation to the null distribution of $T_{CQ}$ will be less adequate or not adequate at all. However, the W--S $\chi^2$-approximation and the  $F$-approximation to the null distribution of $T_{n,p}$ and $F_{n,p}$ are not affected too much. This partially  explains   why $F_{n,p}$ and $T_{n,p}$ outperform $T_{CQ}$ generally. Secondly, it is seen that the values of $\hd_2$ are generally larger than the total sample size $n$ and are generally quite large. This partially explains why $T_{n,p}$ still performs quite well under each setting. Further,  with increasing the total sample size $n$,  the values of $\hd_2$ are also increasing, so that, as seen from Table~\ref{size2.tab}, in terms of size control, the performances of $T_{n,p}$ and $F_{n,p}$ are getting closer and closer. These results are consistent with the conclusions drawn in Remark~\ref{Fapprox.rem}. To save space, we do not present the empirical powers of the three tests under consideration since the conclusions drawn from these empirical powers are similar  to those drawn from Table~\ref{power1.tab}.
	\begin{table}[!h]
		\begin{center}
			\begin{minipage}{0.8\textwidth}
			\caption{Simulation 2: Values of $\hd_1$ and $\hd_2$.}  \label{para2.tab}	
				\begin{tabular*}{\textwidth}{@{\extracolsep{\fill}}p{0.02\textwidth}p{0.01\textwidth}p{0.01\textwidth}p{0.01\textwidth}p{0.01\textwidth}p{0.01\textwidth}p{0.01\textwidth}p{0.01\textwidth}p{0.01\textwidth}@{\extracolsep{\fill}}}
					\hline
					&       &       &        \multicolumn{2}{c}{$\rho_2=0.1$} & \multicolumn{2}{c}{$\rho_2=0.5$} & \multicolumn{2}{c}{$\rho_2=0.9$} \\
					\hline
				Model	& $p$     & $\bn$ & $\hd_1$    & $\hd_2$    & $\hd_1$    & $\hd_2$   & $\hd_1$    & $\hd_2$     \\
					\hline
					\multirow{9}[0]{*}{1} & \multirow{3}[0]{*}{50} & $\bn_1$ & 7     & 452   & 5     & 346   & 4     & 243 \\
					&       & $\bn_2$ & 7     & 1783  & 5     & 1388  & 4     & 969 \\
					&       & $\bn_3$ & 7     & 3564  & 5     & 2778  & 4     & 1938 \\
					& \multirow{3}[0]{*}{500} & $\bn_1$ & 64    & 3932  & 41    & 2832  & 18    & 1061 \\
					&       & $\bn_2$ & 64    & 15962 & 41    & 11507 & 17    & 4173 \\
					&       & $\bn_3$ & 64    & 32018 & 41    & 23081 & 17    & 8305 \\
					& \multirow{3}[0]{*}{1000} & $\bn_1$ & 127   & 7792  & 82    & 5586  & 33    & 1910 \\
					&       & $\bn_2$ & 126   & 31722 & 81    & 22733 & 32    & 7582 \\
					&       & $\bn_3$ & 126   & 63649 & 81    & 45634 & 32    & 15132 \\
						\hline
					\multirow{9}[0]{*}{2} & \multirow{3}[0]{*}{50} & $\bn_1$ & 7     & 431   & 5     & 335   & 4     & 235 \\
					&       & $\bn_2$ & 7     & 1756  & 5     & 1371  & 4     & 959 \\
					&       & $\bn_3$ & 7     & 3530  & 5     & 2757  & 4     & 1927 \\
					& \multirow{3}[0]{*}{500} & $\bn_1$ & 61    & 3656  & 40    & 2676  & 18    & 1032 \\
					&       & $\bn_2$ & 63    & 15531 & 41    & 11279 & 17    & 4131 \\
					&       & $\bn_3$ & 63    & 31554 & 41    & 22811 & 17    & 8271 \\
					& \multirow{3}[0]{*}{1000} & $\bn_1$ & 121   & 7180  & 79    & 5256  & 32    & 1865 \\
					&       & $\bn_2$ & 125   & 30831 & 80    & 22264 & 32    & 7514 \\
					&       & $\bn_3$ & 125   & 62646 & 81    & 45073 & 32    & 15048 \\
						\hline
					\multirow{9}[0]{*}{3} & \multirow{3}[0]{*}{50} & $\bn_1$ & 7     & 424   & 5     & 331   & 4     & 233 \\
					&       & $\bn_2$ & 7     & 1751  & 5     & 1367  & 4     & 959 \\
					&       & $\bn_3$ & 7     & 3532  & 5     & 2758  & 4     & 1929 \\
					& \multirow{3}[0]{*}{500} & $\bn_1$ & 62    & 3659  & 40    & 2684  & 18    & 1040 \\
					&       & $\bn_2$ & 63    & 15661 & 41    & 11337 & 17    & 4150 \\
					&       & $\bn_3$ & 63    & 31707 & 41    & 22908 & 17    & 8289 \\
					& \multirow{3}[0]{*}{1000} &$\bn_1$ & 122   & 7247  & 80    & 5296  & 32    & 1872 \\
					&       & $\bn_2$ & 125   & 31105 & 81    & 22415 & 32    & 7535 \\
					&       & $\bn_3$ & 126   & 63026 & 81    & 45289 & 32    & 15092 \\
					\hline
			\end{tabular*}
		\end{minipage}
	\end{center}
\end{table}

\section{Applications to the COVID-19 Data}\label{appl.sec}

In this section, we apply $T_{CQ}, T_{n,p}$, and $F_{n,p}$ to the COVID-19 data set introduced in Section~\ref{Intro.sec}. It is publicly available at \url{https://www.ncbi.nlm.nih.gov/geo/query/acc.cgi} with accession ID GSE152641. It is of interest  to check whether  the patients with COVID-19 and those  from healthy controls of the COVID-19 data have the same mean transcriptome profiles. Table~\ref{cov.tab1} displays the testing results of applying $T_{CQ}$, $T_{n,p}$,  and $F_{n,p}$ to the COVID-19 data. For easy comparison, the test statistics of $T_{CQ}, T_{n,p}$, and $F_{n,p}$ have been normalized so that they all have mean $0$ and variance $1$.  It is seen that the three  standardized test statistics are comparable and all the three $p$-values are very small and much smaller than $1\%$.  We then  conclude that the mean transcriptome profiles of the two groups of the COVID-19 data  are significantly different.  However, since  the estimated degrees of freedom $\hd_1$ is only $2.73$,  the normal approximation to the null distribution of  $T_{CQ}$ is unlikely to be adequate and hence   the $p$-value of $T_{CQ}$ is not  trustable.
\begin{table}[!h]
	\caption{\em Testing the equality of the mean transcriptom profiles   of the two groups of the  COVID-19 data.\label{cov.tab1}}
	\begin{center}
		\begin{tabular*}{0.8\textwidth}{@{\extracolsep{\fill}}p{0.1\textwidth}p{0.1\textwidth}p{0.1\textwidth}p{0.1\textwidth}p{0.1\textwidth}@{\extracolsep{\fill}}}
				\hline
			Method &Statistic &$p$-value &$\hd_1$ &$\hd_2$\\
			\hline
			$T_{CQ}$ &3.54 &$0.00020$ &- &-\\
			$T_{n,p}$ &3.78 &$0.00693$ &$2.73$ &-\\
			$F_{n,p}$ &3.73 &$0.00867$ &2.73 &171.76\\
					\hline
		\end{tabular*}
	\end{center}
\end{table}
\section{Concluding Remarks}\label{remark.sec}

In this paper, we proposed and studied an $F$-type test for  two-sample BF problems for high-dimensional data. When the two samples are normally distributed, for any given $n$ and $p$, it is easy to see that under the null hypothesis,  the proposed  $F$-type test statistic is  an $F$-type mixture,  a ratio of two independent $\chi^2$-type mixtures. Under some regularity conditions and the null hypothesis, we show that  the proposed $F$-type  test statistic and the  $F$-type mixture have the same normal and  non-normal limits.  It is then justified  to  approximate the null distribution of the proposed $F$-type test statistic using that of  the  $F$-type mixture, resulting in the so-called normal reference $F$-type test. We apply  the Welch--Satterthwaite $\chi^2$-approximation to the distributions of the numerator and the denominator of the $F$-type mixture respectively, resulting the so-called  $F$-approximation to  an  $F$-type mixture.  Simulation studies and a real data example  showed that in terms of size control,  the proposed normal reference $F$-type  test outperforms two existing tests for two-sample BF problems for high-dimensional data.

\section*{Funding and Acknowledgement}

%The authors thank the co-Editor, AE and two reviewers for their constructive comments and suggestions which help us improve the article substantially.
Zhang and Zhu's research  was  partially supported by the National University of Singapore academic research grant R-155-000-212-114 and the National Institute of Education (NIE) start-up grant (NIE-SUG 6-22 ZTM), respectively.
\appendix
\bigskip{}
\centerline{APPENDIX: Technical proofs} \setcounter{equation}{0}
\global\long\def\theequation{A.\arabic{equation}}

\begin{proof}[Proof of Theorem~\ref{limit.thm}] We shall apply Theorems 1 and 2 of \cite{ZhangTSBF2020} for the proof of this theorem. Notice that under Conditions C1--C3, by applying  Theorem 2 of \cite{ZhangTSBF2020}, $\tr(\hbOmega_n)$ is ratio-consistent for $\tr(\bOmega_n)$ uniformly for all $p$. We can write
		\[
		F_{n,p,0}=\frac{T_{n,p,0}}{\tr(\bOmega_n)}[1+o_p(1)],\;\mbox{ and }\;\tF_{n,p,0}=\frac{T_{n,p,0}-\tr(\bOmega_n)}{\sqrt{2\tr(\bOmega_n^2)}}[1+o_p(1)].
		\]
		From (\ref{EVB.equ}), we have $\E(S_{n,p,0}^*)=\tr(\bOmega_n)$ and
		\[ \Var\left[S_{n,p,0}^*/\tr(\bOmega_n)\right]=\frac{2[n_2^{2}(n_1-1)^{-1}\tr(\bSigma_1^2)+n_1^{2}(n_2-1)^{-1}\tr(\bSigma_2^2)]}
		{n_2^{2}\tr^2(\bSigma_1)+n_1^{2}\tr^2(\bSigma_2)+2n_1n_2\tr(\bSigma_1)\tr(\bSigma_2)}.
		\]
		Under Condition C3,	as $n\to\infty$, $ 	\Var\left[S_{n,p,0}^*/\tr(\bOmega_n)\right]\to 0$ uniformly for all $p$. That is, $S_{n,p,0}^*/\tr(\bOmega_n)\convP 1$ 	uniformly for all $p$.  Therefore, we can write	
		\[
		F^*_{n,p,0}=\frac{T^*_{n,p,0}}{\tr(\bOmega_n)}[1+o_p(1)],\;\mbox{ and }\;\tF^*_{n,p,0}=\frac{T^*_{n,p,0}-\tr(\bOmega_n)}{\sqrt{2\tr(\bOmega_n^2)}}[1+o_p(1)].
		\]		
		Then under Conditions C1--C4, as $n,p\to\infty$, Theorem~\ref{limit.thm}(a) and (\ref{supcvg.sec2}) follow directly from Theorem 1(a) of \cite{ZhangTSBF2020}, and under Conditions C1--C3 and C5, as $n,p\to\infty$, Theorem~\ref{limit.thm}(b) and (\ref{supcvg.sec2}) follow directly from Theorem 1(b) of \cite{ZhangTSBF2020}.
		\end{proof}
		%%%%%%%%%%%Theorem 3
	\begin{proof}[Proof of Theorem~\ref{power.thm}]
		We first prove (a). Under Conditions C1--C4, Theorem~\ref{limit.thm}(a) indicates that as $n,p\to\infty$ we have $(F_{n,p,0}-1)/\sqrt{2/d_1}\convL \zeta$. In addition, under Conditions C1--C3, as $n\to\infty$, we have $\hd_1/d_1\convP 1$ and $\hd_2/d_2\convP 1$ uniformly for all $p$. Therefore, under the given conditions, we have
		\[
		\begin{split}
			&\quad\Pr\left[F_{n,p}\geq F_{\hd_1,\hd_2}(\alpha)\right]\\
			&=\Pr\left[F_{n,p,0}\geq F_{\hd_1,\hd_2}(\alpha)-\frac{n_1n_2n^{-1}\|\bmu_1-\bmu_2\|^2}{\tr(\hbOmega_n)}\right][1+o(1)]\\
			&=\Pr\left[\frac{F_{n,p,0}-1}{\sqrt{2/d_1}}\geq \frac{F_{\hd_1,\hd_2}(\alpha)-1}{\sqrt{2/d_1}}-\frac{n_1n_2n^{-1}\|\bmu_1-\bmu_2\|^2}{\sqrt{2/d_1}\tr(\hbOmega_n)}\right][1+o(1)]\\
			&=\Pr\left\{\zeta\geq \frac{F_{d_1,d_2}(\alpha)-1}{\sqrt{2/d_1}}-\frac{n\tau(1-\tau)\|\bmu_1-\bmu_2\|^2}{\left[2\tr(\bOmega^2)\right]^{1/2}}\right\}[1+o(1)].\\
		\end{split}
		\]
		Next we prove (b). Under Conditions C1--C3 and C5, Theorem~\ref{limit.thm}(b) indicates that as $n\to\infty$, we have $(F_{n,p,0}-1)/\sqrt{2/d_1}\convL \calN(0,1)$.
		By Remark~\ref{Fapprox.rem}, we have $[F_{d_1,d_2}(\alpha)-1]/\sqrt{2/d_1}\to z_\alpha$ when $d_2\to\infty$. Therefore, under the given conditions, we have
		\[
		\begin{split}
			&\quad\Pr\left[F_{n,p}\geq F_{\hd_1,\hd_2}(\alpha)\right]\\
			&=\Pr\left[F_{n,p,0}\geq F_{\hd_1,\hd_2}(\alpha)-\frac{n_1n_2n^{-1}\|\bmu_1-\bmu_2\|^2}{\tr(\hbOmega_n)}\right][1+o(1)]\\
			&=\Pr\left[\frac{F_{n,p,0}-1}{\sqrt{2/d_1}}\geq \frac{F_{\hd_1,\hd_2}(\alpha)-1}{\sqrt{2/d_1}}-\frac{n_1n_2n^{-1}\|\bmu_1-\bmu_2\|^2}{\sqrt{2/d_1}\tr(\hbOmega_n)}\right][1+o(1)]\\
			&=\Phi\left\{-z_{\alpha}+\frac{n\tau(1-\tau)\|\bmu_1-\bmu_2\|^2}{\left[2\tr(\bOmega^2)\right]^{1/2}}\right\}[1+o(1)],\\
		\end{split}
		\]
		where $\Phi(\cdot)$ denotes the cumulative distribution of $\calN(0,1)$.
	\end{proof}

\bibliographystyle{apalike}
\bibliography{HDMeanTest}% common bib file


\end{document}

%% file: paperconfig.tex
\usepackage{color}%\textcolor{red}{}
\usepackage{soul,xcolor}
\setstcolor{red}%\st{Some overstruck text}
\newif\ifApproveEdit
\ApproveEditfalse
\ifApproveEdit
  
  \newcommand{\DEL}[1]{\iffalse{#1}\fi}
\else
  
  \newcommand{\DEL}[1]{\st{#1}}
\fi

\usepackage{graphics}%\resizebox{\textwidth}{\textheight or !}{}
\usepackage{epstopdf}
\renewcommand{\baselinestretch}{1.5} % original 1.5

\usepackage{bm}